\documentclass[11pt]{amsart}%
\usepackage{amsmath}
\usepackage{graphicx}
\usepackage{amsfonts}
\usepackage{amssymb}%
\providecommand{\U}[1]{\protect\rule{.1in}{.1in}}
\newtheorem{theorem}{Theorem}[section]
\theoremstyle{plain}

\newtheorem{lemma}{Lemma}[section]

\newtheorem{proposition}{Proposition}[section]

\numberwithin{equation}{section}
\def\e{\varepsilon}

\begin{document}
\title[Regularity for convex solutions]{Regularity for convex viscosity solutions of special Lagrangian equation\\ }
\author{Jingyi Chen}
\address{Department of Mathematics\\
University of British Columbia\\
Vancouver, BC V6T 1Z2}
\email{jychen@math.ubc.ca}
\author{Ravi Shankar}
\author{Yu YUAN}
\address{Department of Mathematics, Box 354350\\
University of Washington\\
Seattle, WA 98195}
\email{shankarr@uw.edu, yuan@math.washington.edu}
\thanks{JYC is partially supported by NSERC Discovery Grant (22R80062) and a grant
(No. 562829) from the Simons Foundation. RS and YY are partially supported by
NSF Graduate Research Fellowship Program under grant No. DGE-1762114 and NSF
grant DMS-1800495 respectively.}
\date{\today}

\begin{abstract}
We establish interior regularity for convex viscosity solutions of the special
Lagrangian equation. Our result states that all such solutions are real
analytic in the interior of the domain.

\end{abstract}
\maketitle

\section{\bigskip Introduction}

In this paper, we establish regularity for convex viscosity solutions of the
special Lagrangian equation%

\begin{equation}
F\left(  D^{2}u\right)  =\sum_{i=1}^{n}\arctan\lambda_{i}-\Theta=0,
\label{EsLag}%
\end{equation}
where $\lambda_{i}^{\prime}s$ are the eigenvalues of the Hessian $D^{2}u$ and
$\Theta$ is constant.

The fully nonlinear equation (\ref{EsLag}) arises from the special Lagrangian
geometry \cite{HL}. The \textquotedblleft gradient\textquotedblright\ graph
$\left(  x,Du\left(  x\right)  \right)  $ of the potential $u$ is a Lagrangian
submanifold in $\mathbb{R}^{n}\times\mathbb{R}^{n}.$ The Lagrangian graph is
called special when the phase, which at each point is the argument of the
complex number $\left(  1+\sqrt{-1}\lambda_{1}\right)  \cdots\left(
1+\sqrt{-1}\lambda_{n}\right)  ,$ is constant $\Theta,$ that is, $u$ satisfies
equation (\ref{EsLag}). A special Lagrangian graph is a volume minimizing
minimal submanifold in $\left(  \mathbb{R}^{n}\times\mathbb{R}^{n}%
,dx^{2}+dy^{2}\right)  .$

A dual form of (\ref{EsLag}) is the Monge-Amp\`{e}re equation%
\begin{equation}
\sum_{i=1}^{n}\ln\lambda_{i}-\Phi=0 \label{MA}%
\end{equation}
with $\Phi$ being constant, interpreted by Hitchin \cite{Hi} as the potential
equation for special Lagrangian submanifolds in $\left(  \mathbb{R}^{n}%
\times\mathbb{R}^{n},dxdy\right)  .$ Warren \cite{W} demonstrated the
\textquotedblleft gradient\textquotedblright\ graph $\left(  x,Du\left(
x\right)  \right)  $ is volume maximizing in this pseudo-Euclidean space.
Under a rotation introduced in \cite{Y2}, (\ref{MA}) becomes%
\begin{equation}
\sum_{i=1}^{n}\ln\frac{1+\lambda_{i}}{1-\lambda_{i}}-\Phi=0. \label{MAR}%
\end{equation}
Earlier on, Mealy \cite{Me} showed that an equivalent algebraic form of
(\ref{MAR}) is the potential equation for his volume maximizing/special
Lagrangian submanifolds in $\left(  \mathbb{R}^{n}\times\mathbb{R}^{n}%
,dx^{2}-dy^{2}\right)  .$

A fundamental problem for those geometrically as well as analytically
significant equations is regularity. Our main result is

\begin{theorem}
\label{theorem1} Let $u$ be a convex viscosity solution of (\ref{EsLag}) on
ball $B_{1}(0)\subset\mathbb{R}^{n}.$ Then $u$ is analytic in $B_{1}\left(
0\right)  $ and we have an effective Hessian bound
\[
|D^{2}u(0)|\leq C(n)\exp\left[  C(n)\operatorname*{osc}_{B_{1}\left(
0\right)  }u\right]  ^{2n-2},
\]
where $C(n)$ is a certain dimensional constant.
\end{theorem}

One application of the above regularity result is that every entire convex
viscosity solution of (\ref{EsLag}) is a quadratic function; the smooth case
was done in [Y2]. In parallel, Caffarelli proved the rigidity for entire
convex viscosity solutions of the Monge-Amp\`{e}re equation, while the smooth
case is the classic work by J\"{o}rgens-Calabi-Pogorelov and also Cheng-Yau.
Another consequence is the existence of interior smooth solutions of the
second boundary value problem for (\ref{EsLag}) between general convex domains
$\Omega$ and $\tilde{\Omega}$ on $\mathbb{R}^{n}.$ For uniformly convex
domains $\Omega$ and $\tilde{\Omega}$ with smooth boundaries, this problem was
solved by Brendle-Warren \cite{BW}. The extension can be handled as follows:
under smooth, uniformly convex approximations of the two general convex
domains, a $C^{0}$ limit of Brendle-Warren solutions is still a convex
viscosity solution of (\ref{EsLag}), and in turn, interior smooth by Theorem
1.1. The boundary behavior of the solutions remains unclear to us. One
by-product of our arguments for the above theorem is that, by Lemma 2.1, we
can remove the local $C^{1,1}$ assumption on the initial convex potential, for
the long time existence of Lagrangian mean curvature flow in [CCY, Theorem 1.2
and Theorem 1.3].

Regularity for two dimensional Monge-Amp\`{e}re type equations including
(\ref{EsLag}) with $n=2$ was achieved by Heinz \cite{H} in the 1950's.
Regularity for continuous viscosity solutions of (\ref{EsLag}) with critical
and supercritical phases $\left\vert \Theta\right\vert \geq\left(  n-2\right)
\frac{\pi}{2}$ follows from the a priori estimates developed in [WY1,2,3]
[CWY] [WdY2]. Singular semiconvex viscosity solutions of (\ref{EsLag})
certainly with subcritical phase $\left\vert \Theta\right\vert <\left(
n-2\right)  \frac{\pi}{2}$ and $n\geq3$ constructed by
Nadirashvili-Vl\u{a}du\c{t} [NV] and in [WdY1], show that the convexity
condition in Theorem 1.1 is necessary. In comparison, there are singular
convex viscosity solutions (Pogorelov $C^{1,1-2/n},$ or more singular ones) to
the Monge-Amp\`{e}re equation $\det D^{2}u=f(x)  .$ Under a
necessary strict convexity assumption on convex viscosity solutions, interior
regularity was obtained respectively by Pogorelov [P] for smooth enough right
hand side $f(x)  ,$ by Urbas [U] for Lipschitz $f(x)$ and
$C^{1,(1-2/n)  ^{+}}$ solutions, and by
Caffarelli [C] for H\"{o}lder $f(x)  .$

There have been attempts for the regularity of convex viscosity solutions of
(\ref{EsLag}) since our work \cite{CWY} in 2009 on a priori estimates for
smooth convex solutions, as in the critical and supercritical cases. In those
latter cases, one can smoothly solve the Dirichlet problem for (now concave
\cite{Y3}) equation (\ref{EsLag}) on any interior small ball, with smooth
boundary data approximating the continuous viscosity solution on the boundary
in $C^{0}$ norm. The a priori estimates in [WY1,3] [WdY2] depend only on
$C^{0}$ norm of the $C^{0}$ viscosity solution on the boundary, thus allow one
to draw a smooth limit to the $C^{0}$ viscosity solution. Hence, the
regularity for viscosity solutions follows. We are not able to find smooth
convex solutions of (now saddle \cite{Y3}) equation \eqref{EsLag} of
subcritical phase $\Theta<(n-2)\frac{\pi}{2},$ with smooth boundary data
approximating the convex viscosity solution there in $C^{0}$ norm. Even if we
solve the Dirichlet problem for a modified concave equation $\hat{F}%
(D^{2}u)=\sum_{1}^{n}f(\lambda_{i})-\Theta=0$ with $f(\lambda)=\arctan\lambda$
for $\lambda\geq0$ and $\lambda$ for $\lambda<0,$ the smooth solutions with
the approximated smooth boundary data may not be convex. Unless one proves
similar a priori estimates directly for the modified equation, we cannot gain
regularity by drawing a smooth limit to the original convex viscosity
solution, with the a priori estimates for smooth convex special Lagrangian
solutions in [CWY].

Another natural way is to work over a rotated coordinate system introduced in
\cite{Y2}, so that the slope of the \textquotedblleft
gradient\textquotedblright\ graph of the solution drops to the range $[-1,1]$
from $[0,+\infty].$ As every graphical tangent cone to the minimal Lagrangian
graph with such restricted slopes is flat, via the machinery from geometric
measure theory, \cite{Y1} gives a $C^{2,\alpha}$ interior bound, hence
regularity in rotated coordinates.

The first difficulty is in dealing with the multivalued \textquotedblleft
gradient\textquotedblright\ graph over the rotated coordinate system. We
relate the above $\frac{\pi}{4}$-rotation to the conjugated $\frac{\pi}{2}%
$-rotation, that is, at potential level, we rewrite the rotated potential in
terms of the Legendre transform (convex conjugate) of the old potential
(Proposition 2.1, Lemma 2.1. Geometrically, the \textquotedblleft
singular\textquotedblright\ multivalued gradient \textquotedblleft
graph\textquotedblright\ is a Lipschitz one in the rotated coordinates. This
leads to yet another proof for the Alexandrov Theorem: every semiconvex
function is second order differentiable almost everywhere.). The second one is
showing the rotated potential is still a viscosity solution of \eqref{EsLag}
(with a decreased phase). The preservation of supersolutions is simple because
of the order preservation and the respect for uniform convergence of the
rotation operation (Proposition 2.2). The preservation of subsolutions under
the rotations is no quick matter. Unlike in the supersolution case, we are
only able to show the preservation of convex subsolutions, by convex smooth
subsolution approximations of the original convex subsolution of (now concave)
equation \eqref{EsLag} (Proposition 2.3). There is one last hurdle in making
sure the slope of the \textquotedblleft gradient\textquotedblright\ graph over
the $\frac{\pi}{4}$-rotated coordinate is not $1,$ the largest possible.
Otherwise the original potential cannot have bounded Hessian. It turns out the
maximum eigenvalue of the rotated Hessian is a subsolution of the linearized
equation of the now saddle equation (\ref{EsLag}). The strong maximum
principle then finishes the job; see Section 3.

\section{Preliminaries}

\subsection{Smooth functions and solutions}

Via the Legendre transform, we directly connect the original potential $u$ for
the Lagrangian graph $\left(  x,Du\left(  x\right)  \right)  $ on
$\mathbb{R}^{n}\times\mathbb{R}^{n}=\mathbb{C}^{n},$ with the $\alpha$-rotated
potential $\bar{u}\left(  \bar{x}\right)  $ for the same Lagrangian
submanifold $\left(  \bar{x},Du\left(  \bar{x}\right)  \right)  $ on
$\mathbb{C}^{n},$ under (anti-clockwise $\alpha\in\left(  0,\frac{\pi}%
{2}\right)  $) coordinate rotation $\bar{z}=e^{-i\alpha}z.$ As in [Y2, p.
124], assuming semiconvexity $D^{2}u>-\cot\alpha\ I$ and denoting $\left(
c,s\right)  =\left(  \cos\alpha,\sin\alpha\right)  ,$ the two
\textquotedblleft gradients\textquotedblright\ are related by%

\[
\left\{
\begin{array}
[c]{c}%
\bar{x}=cx+sDu\left(  x\right) \\
D\bar{u}\left(  \bar{x}\right)  =-sx+cDu\left(  x\right)
\end{array}
\right.  .
\]
Writing the original \textquotedblleft gradient\textquotedblright%
\ $u_{x}\left(  x\right)  =Du\left(  x\right)  $ in terms of the old
independent $x$ and the new variable $\bar{x}=cx+su_{x},$ and applying the
product rule, one has the \textquotedblleft gradient\textquotedblright%
\ connection%
\begin{align*}
d\bar{u}\left(  \bar{x}\right)   &  =\bar{u}_{\bar{x}}d\bar{x}=\left(
-sx+cu_{x}\right)  d\bar{x}\\
&  =\left[  -sx+c\frac{\bar{x}-cx}{s}\right]  d\bar{x}=\left[  \frac{c}{s}%
\bar{x}-\frac{1}{s}x\right]  d\bar{x}\\
&  =d\left\{  \frac{1}{2}\frac{c}{s}\bar{x}^{2}-\frac{1}{s}\left[  \bar
{x}x-\left(  su\left(  x\right)  +\frac{c}{2}x^{2}\right)  \right]  \right\}
,
\end{align*}
and up to a constant, in formal notation instead of the above
\textquotedblleft abused\textquotedblright\ one, the potential connection, as
in [CW, p.334-335],%
\[
\bar{u}\left(  \bar{x}\right)  =\frac{1}{2}\frac{c}{s}\left\vert \bar
{x}\right\vert ^{2}-\frac{1}{s}\left[  \bar{x}\cdot x-\left(  su\left(
x\right)  +\frac{c}{2}\left\vert x\right\vert ^{2}\right)  \right]  .
\]

When $\alpha=\frac{\pi}{2},$ the rotated potential $\bar{u}\left(  \bar
{x}\right)  =-\left[  \bar{x}\cdot x-u\left(  x\right)  \right]  $ is the
negative of the Legendre transform of $u\left(  x\right)  .$ Recall the
Legendre transform of a strictly convex (not necessarily smooth) function
$f\left(  x\right)  $ on $B_{1}$ is usually formulated in an extremal form
\begin{equation}
f^{\ast}(y)=\sup_{x\in B_{1}}\left[  y\cdot x-f(x)\right]  \label{Legendre}%
\end{equation}
for $y\in\partial f\left(  B_{1}\right)  .$ We record the following analytic
interpretation of the $\alpha$-rotated potential in terms of the Legendre
transform of the original one.

\begin{proposition}
Suppose $u(x)$ is smooth and $D^{2}u>-\cot\alpha\ I$ on $B_{1}.$ Then the
smooth function%
\begin{equation}
\bar{u}\left(  \bar{x}\right)  =\frac{1}{2}\frac{c}{s}\left\vert \bar
{x}\right\vert ^{2}-\frac{1}{s}\left[  su\left(  x\right)  +\frac{c}%
{2}\left\vert x\right\vert ^{2}\right]  ^{\ast}\left(  \bar{x}\right)
\label{a-rotation}%
\end{equation}
for $\bar{x}\in\left[  cx+sDu\left(  x\right)  \right]  \left(  B_{1}\right)
,$ is a potential for the Lagrangian graph $\left(  x,Du\left(  x\right)
\right)  $ under the anti-clockwise coordinate rotation $\bar{z}=e^{-i\alpha
}z.$
\end{proposition}

Observe that the canonical angles between each tangent plane of the Lagrangian
graph $\left(  x,Du\left(  x\right)  \right)  $ and the $\alpha$-rotated
$\bar{x}$-plane, decrease from the original ones with respect to the $x$-plane
by $\alpha$%
\[
\arctan\bar{\lambda}_{i}\left(  D^{2}\bar{u}\right)  =\arctan\lambda
_{i}\left(  D^{2}u\right)  -\alpha;
\]
consequently we see the $\alpha$-rotated potential $\bar{u}$ satisfies the
equation%
\begin{equation}
\bar{F}\left(  D^{2}\bar{u}\right)  =\sum_{i=1}^{n}\arctan\bar{\lambda}%
_{i}\left(  D^{2}\bar{u}\right)  -\Theta+n\alpha=0, \label{Ea-sLag}%
\end{equation}
given the equation (\ref{EsLag}) for $u.$

\subsection{Convex functions and viscosity solutions}

The $\alpha$-rotation $\bar{u}$ in (\ref{a-rotation}) still makes sense if
$u\left(  x\right)  +\frac{1}{2}\cot\alpha$ $\left\vert x\right\vert ^{2}$ is
strictly convex (not necessarily smooth). Indeed, we have

\begin{lemma}
Let $u+\frac{1}{2}\left(  \cot\alpha-\delta\right)  \left\vert x\right\vert
^{2}$ be convex on $B_{1}\left(  0\right)  $ for $\delta>0.$ Then the $\alpha
$-rotation $\bar{u}$ in (\ref{a-rotation}) is defined on $\bar{\Omega
}=\partial\tilde{u}\left(  B_{1}\left(  0\right)  \right)  $ with $\tilde
{u}=su+\frac{c}{2}\left\vert x\right\vert ^{2}.$ Moreover, $\bar{\Omega}$ is
open and connected.
\end{lemma}

\begin{proof}
Step 1. The subdifferential $\partial f\left(  a\right)  $ for any convex
function $f\left(  x\right)  $ at $a$ is the set of all the gradients of those
linear support functions for $f\left(  x\right)  $ at $x=a.$ Recall $\partial
f\left(  a\right)  $ is bounded, closed, and convex. The convexity of
$\partial f\left(  a\right)  $ is because convex combinations of those linear
support functions remain linear support ones at $x=a.$ The sum rule is valid
for subdifferentials of two convex functions; see [R,Theorem 23.8]. For
completeness, we include a proof of the (still subtle) sum rule with the other
convex function being a quadratic one%
\[
\partial\left(  v+Q\right)  \left(  x\right)  =\partial v\left(  x\right)
+\partial Q\left(  x\right)  .
\]

For the sake of simple notation, we only present the proof at $x=0.$ By
subtracting linear supporting functions from $v$ and $Q$ at $x=0,$ we assume
that $v\left(  0\right)  =0,$ $v\geq0,$ $\vec{0}\in\partial v\left(  0\right)
,$ and $Q\left(  x\right)  =0.5\kappa\left\vert x\right\vert ^{2}$ with
$\kappa>0.$ The inclusion $\partial\left(  v+Q\right)  \left(  0\right)
\supseteq\partial v\left(  0\right)  +\partial Q\left(  0\right)  $ is easy,
because the sum of any two linear supporting functions at the same point for
two convex functions, is still a linear supporting function for the sum
function at that same point. On the other hand, for any $Y\in\partial\left(
v+Q\right)  \left(  0\right)  ,$ we show $Y\in\partial v\left(  0\right)
+\partial Q\left(  0\right)  =\partial v\left(  0\right)  .$ Otherwise, even
$\left(  1-\eta\right)  Y$ for a small $\eta>0$ is not in the bounded closed
convex set $\partial v\left(  0\right)  \ni\vec{0}.$ This means non-vanishing
linear function $\left(  1-\eta\right)  Y\cdot x$ would be larger than
$v\left(  x\right)  ,$ along a sequence $x_{\gamma}$ going to $0$ with $Y\cdot
x_{\gamma}>0.$ In turn%
\[
\left(  1-\eta\right)  Y\cdot x_{\gamma}+0.5\kappa\left\vert x_{\gamma
}\right\vert ^{2}\geq v\left(  x_{\gamma}\right)  +0.5\kappa\left\vert
x_{\gamma}\right\vert ^{2}\geq Y\cdot x_{\gamma}.
\]
Then $0.5\kappa\left\vert x_{\gamma}\right\vert ^{2}\geq\eta Y\cdot x_{\gamma
}>0.$ Impossible for small $x_{\gamma}.$

Step 2. We first prove $\bar{\Omega}$ contains $\bar{B}_{s\delta}\left(
\partial\tilde{u}\left(  0\right)  \right)  .$ For any subdifferential
$\bar{x}_{0}\in$ $\partial\tilde{u}\left(  0\right)  ,$ by subtracting linear
function $\bar{x}_{0}\cdot x+\tilde{u}\left(  0\right)  $ from $\tilde{u},$ we
assume $\bar{x}_{0}=0\in\partial\tilde{u}\left(  0\right)  $ and $\tilde
{u}\left(  0\right)  =0,\ $then $\tilde{u}\geq0$ in $B_{1}\left(  0\right)  .$
For any $\left\vert \bar{x}_{\ast}\right\vert <s\delta,$ the linear function
$L\left(  x\right)  =\bar{x}_{\ast}\cdot\left(  x-\frac{1}{s\delta}\bar
{x}_{\ast}\right)  +\frac{1}{2s\delta}\left\vert \bar{x}_{\ast}\right\vert
^{2}$ touches $Q\left(  x\right)  =\frac{s\delta}{2}\left\vert x\right\vert
^{2}$ from below at $x=\frac{1}{s\delta}\bar{x}_{\ast}.$ Because all the
directional derivatives of $\tilde{u}-Q$ are nonnegative at $0,$ and also the
directional derivative of $\tilde{v}-Q$ along each ray from $0$ is increasing,
we have the ordering $\tilde{u}\left(  x\right)  \geq Q\left(  x\right)  \geq
L\left(  x\right)  .$ We move up the linear function $L\left(  x\right)  $
until it touches the graph of $\tilde{u}\left(  x\right)  $ in $\mathbb{R}%
^{n}\times\mathbb{R}^{1}$ the first time at $x=b.$ Without loss of generality,
we assume that in the first place, $\tilde{u}$ is already extended to an
entire $s\delta$-convex function on $\mathbb{R}^{n}.$ Note that the point $b$
cannot be outside $B_{1}\left(  0\right)  .$ Otherwise, by the sum rule in
Step 1, $\bar{x}_{\ast}\in\partial\tilde{u}\left(  b\right)  =\partial
v\left(  b\right)  +s\delta\cdot b$ and $0\in\partial\tilde{u}\left(
0\right)  =\partial v\left(  0\right)  +\partial Q\left(  0\right)  =\partial
v\left(  0\right)  $ with $v=\tilde{u}-Q.$ Then we have the slope increasing
property for $\tilde{u}:$
\[
\left\vert \bar{x}_{\ast}-0\right\vert ^{2}=\left\vert \partial v\left(
b\right)  +s\delta b\right\vert ^{2}=\left\vert \partial v\left(  b\right)
\right\vert ^{2}+\left\vert s\delta b\right\vert ^{2}+2\left\langle \partial
v\left(  b\right)  ,s\delta b\right\rangle \geq\left\vert s\delta b\right\vert
^{2}\geq\left\vert s\delta\right\vert ^{2},
\]
where the \textquotedblleft abused\textquotedblright\ notation $\partial
v\left(  b\right)  \ $means $\bar{x}_{\ast}-s\delta b,$ and the inequality
$\left\langle \partial v\left(  b\right)  ,b\right\rangle \geq0$ comes from
the summation of the following two for convex function $v$%
\begin{align*}
v\left(  b\right)  -v\left(  0\right)   &  \geq\left\langle 0,b-0\right\rangle
,\\
v\left(  0\right)  -v\left(  b\right)   &  \geq\left\langle \partial v\left(
b\right)  ,0-b\right\rangle .
\end{align*}
It contradicts $\left\vert \bar{x}_{\ast}\right\vert <s\delta.$ Thus
$\bar{\Omega}$ contains $\bar{B}_{s\delta}\left(  \partial\tilde{u}\left(
0\right)  \right)  .$

Similarly $\bar{\Omega}$ contains $\bar{B}_{s\delta\left(  1-\left\vert
a\right\vert \right)  }\left(  \partial\tilde{u}\left(  a\right)  \right)  $
for all $a\in B_{1}\left(  0\right)  ,$ and in turn, as a union of those open
sets, $\bar{\Omega}$ is open.

Lastly, the connectedness of the $\bar{\Omega}=\partial\tilde{u}\left(
B_{1}\left(  0\right)  \right)  $ follows from the continuity of the mapping
$\partial\tilde{u}:B_{1}\left(  0\right)  \rightarrow\bar{\Omega}$ in the
sense that, given any $b\in B_{1}\left(  0\right)  ,$ for any $\varepsilon>0,$
there exists $\eta>0$ such that the convex, thus connected subdifferential
$\partial\tilde{u}\left(  x\right)  $ satisfies%
\[
\partial\tilde{u}\left(  a\right)  \subset B_{\varepsilon}\left(
\partial\tilde{u}\left(  b\right)  \right)  \ \ \ \text{for all }\left\vert
a-b\right\vert <\eta;
\]
see [R, Corollary 24.5.1].
\end{proof}

In order to proceed further, we observe some simple key facts. The Legendre
transform (\ref{Legendre}) is order reversing and respects constants: $f\leq
g\rightarrow f^{\ast}\geq g^{\ast},$ and $(f+c)^{\ast}=f^{\ast}-c.$ In
particular, if $f-c\leq g\leq f+c$, then $f^{\ast}+c\geq g^{\ast}\geq f^{\ast
}-c,$ so the transform respects $C^{0}$ uniform convergence. Consequently, the
$\alpha$-rotation (\ref{a-rotation}) also enjoys these three properties,
except now the order is preserved: if $u-c\leq v\leq u+c,$ then $\bar{u}%
-c\leq\bar{v}\leq\bar{u}+c.$

As an immediate application of the uniform respect for the $\alpha$-rotation,
by taking smooth and $\cot\alpha$-semiconvex approximations of the $\cot
\alpha$-semiconvex function $u\left(  x\right)  ,$ we see $\bar{u}$ is
$C^{1,1}$ from above, and if $u\left(  x\right)  $ is $\left(  \cot
\alpha-\delta\right)  $-semiconvex, also $C^{1,1}$ from below%
\[
-K\left(  \alpha,\delta\right)  \ I\leq D^{2}\bar{u}\leq\cot\alpha\ I.
\]

A quick consequence of the order preservation is the preservation of the
supersolutions under the $\alpha$-rotation.

\begin{proposition}
Let $u+\frac{1}{2}\left(  \cot\alpha-\delta\right)  $ $\left\vert x\right\vert
^{2}$ be convex and $u$ be a viscosity supersolution of (\ref{EsLag}) on
$B_{1}\left(  0\right)  .$ Then the $\alpha$-rotation $\bar{u}$ in
(\ref{a-rotation}) is a corresponding viscosity supersolution of
(\ref{Ea-sLag}) on open $\bar{\Omega}=\partial\tilde{u}\left(  B_{1}\left(
0\right)  \right)  $ with $\tilde{u}=su+\frac{c}{2}\left\vert x\right\vert
^{2}.$
\end{proposition}

\begin{proof}
Let $\bar{Q}$ be any quadratic function touching $\bar{u}$ from below locally
somewhere on the open set $\bar{\Omega},$ say the origin. Already $D^{2}%
\bar{Q}\leq D^{2}\bar{u}\leq$ $\cot\alpha\ I.$ By subtracting $\varepsilon
\left\vert \bar{x}\right\vert ^{2}$ from $\bar{Q},$ then taking the limit as
$\varepsilon$ goes to $0,$ we assume $D^{2}\bar{Q}<\cot\alpha\ I.$ This
guarantees the existence of its pre-rotated quadratic function $Q.$ From the
order preservation of $\alpha$-rotation, which is also valid for any reverse
rotation, we see the pre-rotated quadratic function $Q$ touches $u$ from below
somewhere on $B_{1}\left(  0\right)  .$ Because $u$ is a supersolution there,
$\sum_{i=1}^{n}\arctan\lambda_{i}\left(  D^{2}Q\right)  \leq\Theta,$ and in
turn, $\bar{F}\left(  D^{2}\bar{Q}\right)  =\sum_{i=1}^{n}\arctan\bar{\lambda
}_{i}\left(  D^{2}\bar{Q}\right)  -\Theta+n\alpha\leq0.$
\end{proof}

The preservation of subsolutions under $\alpha$-rotation is no quick matter.
Let $\bar{Q}$ be a quadratic function touching $\bar{u}$ from above. When one,
but not all, of the eigenvalues is largest possible $\cot\alpha,$ we are
unable to check $\bar{F}\left(  D^{2}\bar{Q}\right)  \geq0.$ In this scenario,
one cannot \textquotedblleft lower\textquotedblright\ $\bar{Q}$ so that all
the eigenvalues are strictly less than $\cot\alpha$ and above $\bar{u}$ at the
same time. The pre-rotated function $Q$ touching $u$ from above is not a
quadratic function anymore. It is a cone in some subspace, and only quadratic
in the complementary subspace. One cannot see $F\left(  D^{2}Q\right)  \geq0.$
In fact, $u$ is not \ $C^{1,1}$ from above at this touching point; the very
definition of viscosity subsolution requires no checking at such points (of no
touching by quadratic functions), and in turn, gives no information on $Q.$
Moreover, we are unable to show that the points at such $\bar{Q}$ touching
$\bar{u}$ from above have zero measure. Otherwise, the $C^{1,1}$ function
$\bar{u}$ is readily a subsolution.

We are only able to show the preservation of convex subsolutions by convex
smooth subsolution approximations.

\begin{proposition}
Let $u$ be a convex viscosity subsolution of (\ref{EsLag}) on $B_{1.2}\left(
0\right)  .$ Then the (anti-clockwise $\alpha\in\left(  0,\frac{\pi}%
{2}\right)  $) $\alpha$-rotation $\bar{u}$ in (\ref{a-rotation}) is a
corresponding viscosity subsolution of (\ref{Ea-sLag}) on open and connected
$\bar{\Omega}=\partial\tilde{u}\left(  B_{1}\left(  0\right)  \right)  $ with
$\tilde{u}\left(  x\right)  =\frac{1}{2}c\left\vert x\right\vert
^{2}+su\left(  x\right)  .$
\end{proposition}

\begin{proof}
Step 1. By Lemma 2.1 with $\delta=c/s$ and the sum rule $\partial\tilde
{u}=cx+s\partial u\left(  x\right)  ,$ we see that $\alpha$-rotation $\bar{u}$
is indeed defined on the open and connected set $\left(  cx+s\partial u\left(
x\right)  \right)  \left(  B_{1.2}\left(  0\right)  \right)  .$ For
convenience, we extend the convex $u\left(  x\right)  $ to an entire convex
function on $\mathbb{R}^{n}.$ Set the standard convolution $u_{\varepsilon
}\left(  x\right)  =u\ast\rho_{\varepsilon}\left(  x\right)  $ with
$\rho_{\varepsilon}\left(  x\right)  =\varepsilon^{-n}\rho\left(
x/\varepsilon\right)  $ and nonnegative $\rho\left(  x\right)  =\rho\left(
\left\vert x\right\vert \right)  \in C_{0}^{\infty}\left(  \mathbb{R}%
^{n}\right)  $ satisfying $\int_{\mathbb{R}^{n}}\rho\left(  x\right)  dx=1.$
Given the $C^{0}$ uniform continuity of $u,$ we have $\left\vert
u_{\varepsilon}\left(  x\right)  -u\left(  x\right)  \right\vert <o\left(
1\right)  $ for all small enough $\varepsilon.$

We claim that the smooth $\alpha$-rotation $\bar{u}_{\varepsilon}$ is defined
at least on $\bar{\Omega}$ for all small enough $\varepsilon.$ We verify this
by showing that for any $\bar{a}\in\partial\tilde{u}\left(  a\right)  $ with
$a\in B_{1}\left(  0\right)  ,$ there exists $b$ such that $D\tilde
{u}_{\varepsilon}\left(  b\right)  =\bar{a}$ with $\tilde{u}_{\varepsilon
}\left(  x\right)  =\frac{1}{2}c\left\vert x\right\vert ^{2}+su_{\varepsilon
}\left(  x\right)  $ and $\left\vert b-a\right\vert \leq o\left(  1\right)  $
as $\varepsilon$ goes to $0.$ Consequently, $\partial\tilde{u}\left(
B_{1}\left(  0\right)  \right)  \subset D\tilde{u}_{\varepsilon}\left(
B_{1.1}\left(  0\right)  \right)  $ for all small enough $\varepsilon.$

Now for any $\bar{a}\in\partial\tilde{u}\left(  a\right)  ,$ given the uniform
convexity of $\tilde{u}_{\varepsilon},$ $D^{2}\tilde{u}_{\varepsilon}\geq cI,$
there exists $b\in\mathbb{R}^{n}$ such that $D\tilde{u}_{\varepsilon}\left(
b\right)  =\bar{a}.$ By subtracting linear function $\bar{a}\cdot x$ from both
$\tilde{u}$ and $\tilde{u}_{\varepsilon},$ we assume $0\in\partial\tilde
{u}\left(  a\right)  \cap\partial\tilde{u}_{\varepsilon}\left(  b\right)  .$
Then coupled with the $c$-convexity of $\tilde{u}$ and $\tilde{u}%
_{\varepsilon},$ we have%
\[
\tilde{u}\left(  b\right)  -\tilde{u}\left(  a\right)  \geq\frac{c}%
{2}\left\vert b-a\right\vert ^{2}\ \ \text{and\ }\tilde{u}_{\varepsilon
}\left(  a\right)  -\tilde{u}_{\varepsilon}\left(  b\right)  \geq\frac{c}%
{2}\left\vert a-b\right\vert ^{2}.
\]
For small enough $\varepsilon,$ we always have%
\[
\tilde{u}\left(  a\right)  -\tilde{u}_{\e}\left(  a\right)  \geq-\left\vert
o\left(  1\right)  \right\vert \ \ \ \text{and \ }\tilde{u}_{\e}(b)-\tilde{u}\left(  b\right)  \geq-\left\vert o\left(  1\right)  \right\vert
.
\]
Adding all the above four inequalities together, we get%
\[
\left\vert b-a\right\vert ^{2}\leq2\left\vert o\left(  1\right)  \right\vert
/c.
\]
for small enough $\varepsilon.$ Therefore, we have proved that $\bar
{u}_{\varepsilon}$ is defined on $\bar{\Omega}=\partial\tilde{u}\left(
B_{1}\left(  0\right)  \right)  \subset D\tilde{u}_{\varepsilon}\left(
B_{1.1}\left(  0\right)  \right)  $ for all small enough $\varepsilon.$

Step 2. Note that the equation (\ref{EsLag}) is concave for convex $u.$ By the
well-known result in [CC, p. 56], the solid convex average $u\ast
\rho_{\varepsilon}$ (instead of the hollow spherical one there) is still a
subsolution of (\ref{EsLag}) in $B_{1.1}\left(  0\right)  $ for small enough
$\varepsilon>0.$ For smooth convex subsolutions $u_{\varepsilon},$ the
corresponding smooth $\alpha$-rotation $\bar{u}_{\varepsilon}$ is a
subsolution of (\ref{Ea-sLag}) on $\bar{\Omega}$ from Step 1 and the end of
Section 2.1. The viscosity solutions are stable under $C^{0}$ uniform
convergence. Hence uniformly convergent limit $\lim_{\varepsilon\rightarrow
0}\bar{u}_{\varepsilon}=\bar{u}$ is a viscosity subsolution of (\ref{Ea-sLag})
on $\bar{\Omega}.$
\end{proof}

\section{Proof of the theorem}

By Proposition 2.2 and Proposition 2.3 with $\alpha=\frac{\pi}{4}$ and
$\delta=1,$ the $\frac{\pi}{4}$-rotation $\bar{u}$ is a viscosity solution of
(\ref{Ea-sLag}) on open and connected set $\bar{\Omega}=\left(  cx+s\partial
u\left(  x\right)  \right)  \left(  B_{1}\left(  0\right)  \right)  $ (we may
assume $u$ is defined on $B_{1.2}\left(  0\right)  $ by scaling $1.2^{2}%
u\left(  x/1.2\right)  $ ). By the argument before Proposition 2.2, we have%
\[
-I\leq D^{2}\bar{u}\leq I.
\]

Step 1. We now claim $\bar{u}$ is smooth by modifying the dimension-3 interior
$C^{2,\alpha}$ a priori estimate arguments in \cite{Y1}, using the a priori
calculation in \cite{Y2}, Proposition 2.1. We can repeat the proof of Lemma
2.2 in \cite{Y1} almost verbatim, and blow up the Lipschitz minimal surface
$(\bar{x},D\bar{u}(\bar{x}))$ at any point $\bar{p}\in\bar{\Omega}$ to produce
a graphical minimal tangent cone $T$ at $(\bar{p},D\bar{u}(\bar{p}))$
satisfying the same Hessian bounds. If $T$ is not smooth away from the origin,
then by the dimension reduction argument, we produce a graphical minimal cone
$C$ smooth away from its vertex at $(\bar{p},D\bar{u}(\bar{p}))$ which
satisfies the same Hessian bounds. Now, invoking the Hessian bounds and
Proposition 2.1 in \cite{Y2}, we conclude $C$ is flat, hence that $T$ is
smooth away from its vertex as in \cite{Y1}. Proposition 2.1 in \cite{Y2} then
implies that $T$ is flat. The rest of the proof of Theorem 1.1 in \cite{Y1}
now goes through. We conclude $\bar{u}\in C^{2,\alpha}(\bar{\Omega})$, hence
that the special Lagrangian submanifold $(\bar{x},D\bar{u}(\bar{x}))$ for
$\bar{x}\in\bar{\Omega}$ is smooth and even analytic, by the classical
elliptic theory (cf. Theorem 17.16 in \cite{GT} and \cite{M} p. 203).

Step 2. Next we show the strict inequality $D^{2}\bar{u}<I$ on the open and
connected set $\bar{\Omega},$ which then implies that the original $u$
satisfies $D^{2}u<+\infty,$ and hence is smooth and even analytic on
$B_{1}\left(  0\right)  .$

Instead of invoking Lemma 4.1 in [CCY], we give another simple argument.
Otherwise, there exists $\bar{p}\in\bar{\Omega}$ such that $1=\bar{\lambda
}_{\max}=\bar{\lambda}_{1}=\cdots=\bar{\lambda}_{m}>\bar{\lambda}_{m+1}%
\geq\cdots\geq\bar{\lambda}_{n}\geq-1$ at $\bar{p},$ where $\bar{\lambda}%
_{i}^{\prime}s$ are the eigenvalues of $D^{2}\bar{u}.$ We claim that the
Lipschitz function $\bar{\lambda}_{\max}$ is subharmonic, or rather the smooth
function
\[
b_{m}=\frac{1}{m}\sum_{i=1}^{m}\ln\sqrt{1+\bar{\lambda}_{i}^{2}}%
\]
satisfies $\bigtriangleup_{\bar{g}}b_{m}\geq0\ \ \ $near $\bar{p},$ where
$\bar{g}=I+D^{2}\bar{u}D^{2}\bar{u}$ is the induced metric of $\left(  \bar
{x},D\bar{u}\left(  \bar{x}\right)  \right)  $ in $\mathbb{R}^{n}%
\times\mathbb{R}^{n}.$ By the formula in [WdY, p. 487 ] with all the
coefficients for $h_{ijk}^{2}$ re-arranged as sums of nonnegative terms for
$1\geq\bar{\lambda}_{1}\geq\cdots\geq\bar{\lambda}_{m}>\bar{\lambda}_{m+1}%
\geq\cdots\geq\bar{\lambda}_{n}\geq-1$ near $\bar{p},$ we have%

\begin{align*}
m\bigtriangleup_{\bar{g}}b_{m}  &  =%
%TCIMACRO{\tsum \limits_{\gamma=1}^{m}}%
%BeginExpansion
{\textstyle\sum\limits_{\gamma=1}^{m}}
%EndExpansion
\bigtriangleup_{\bar{g}}\ln\sqrt{1+\bar{\lambda}_{\gamma}^{2}}\\
&  =%
%TCIMACRO{\tsum \limits_{k\leq m}}%
%BeginExpansion
{\textstyle\sum\limits_{k\leq m}}
%EndExpansion
\left(  1+\bar{\lambda}_{k}^{2}\right)  {\small h}_{kkk}^{2}{\small +}(%
%TCIMACRO{\tsum \limits_{i<k\leq m}}%
%BeginExpansion
{\textstyle\sum\limits_{i<k\leq m}}
%EndExpansion
+%
%TCIMACRO{\tsum \limits_{k<i\leq m}}%
%BeginExpansion
{\textstyle\sum\limits_{k<i\leq m}}
%EndExpansion
)\left(  3+\bar{\lambda}_{i}^{2}+2\bar{\lambda}_{i}\bar{\lambda}_{k}\right)
{\small h}_{iik}^{2}\ \ \ \\
&  +%
%TCIMACRO{\tsum \limits_{k\leq m<i}}%
%BeginExpansion
{\textstyle\sum\limits_{k\leq m<i}}
%EndExpansion
\frac{2\bar{\lambda}_{k}\left(  1+\bar{\lambda}_{k}\bar{\lambda}_{i}\right)
}{\bar{\lambda}_{k}-\bar{\lambda}_{i}}{\small h}_{iik}^{2}+%
%TCIMACRO{\tsum \limits_{i\leq m<k}}%
%BeginExpansion
{\textstyle\sum\limits_{i\leq m<k}}
%EndExpansion
\frac{\bar{\lambda}_{i}-\bar{\lambda}_{k}+\bar{\lambda}_{i}^{2}\left(
2+\bar{\lambda}_{i}^{2}+\bar{\lambda}_{i}\bar{\lambda}_{k}\right)  }%
{\bar{\lambda}_{i}-\bar{\lambda}_{k}}h_{iik}^{2}\\
&  +2%
%TCIMACRO{\tsum \limits_{i<j<k\leq m}}%
%BeginExpansion
{\textstyle\sum\limits_{i<j<k\leq m}}
%EndExpansion
\left(  3+\bar{\lambda}_{i}\bar{\lambda}_{j}+\bar{\lambda}_{j}\bar{\lambda
}_{k}+\bar{\lambda}_{k}\bar{\lambda}_{i}\right)  h_{ijk}^{2}\\
&  +2%
%TCIMACRO{\tsum \limits_{i<j\leq m<k}}%
%BeginExpansion
{\textstyle\sum\limits_{i<j\leq m<k}}
%EndExpansion
\left(  1+\bar{\lambda}_{i}\bar{\lambda}_{j}+\bar{\lambda}_{i}\frac
{1+\bar{\lambda}_{i}\bar{\lambda}_{k}}{\bar{\lambda}_{i}-\bar{\lambda}_{k}%
}+\bar{\lambda}_{j}\frac{1+\bar{\lambda}_{j}\bar{\lambda}_{k}}{\bar{\lambda
}_{j}-\bar{\lambda}_{k}}\right)  h_{ijk}^{2}\\
&  +2%
%TCIMACRO{\tsum \limits_{i\leq m<j<k}}%
%BeginExpansion
{\textstyle\sum\limits_{i\leq m<j<k}}
%EndExpansion
\bar{\lambda}_{i}\left(  \frac{1+\bar{\lambda}_{i}\bar{\lambda}_{j}}%
{\bar{\lambda}_{i}-\bar{\lambda}_{j}}+\frac{1+\bar{\lambda}_{j}\bar{\lambda
}_{k}}{\bar{\lambda}_{j}-\bar{\lambda}_{k}}\right)  h_{ijk}^{2}\\
&  \geq0\ \ \ \ \ \text{near }\bar{p}.
\end{align*}
By the strong maximum principle $\bar{\lambda}_{\max}\equiv1$ everywhere on
the connected open set $\bar{\Omega}=\partial\tilde{u}\left(  B_{1}\left(
0\right)  \right)  .$ Note that the constant rank result in [CGM, p. 1772]
does not apply here, as our smooth solution $\bar{u}$ with $-I\leq D^{2}%
\bar{u}\leq I$ cannot be turned into a smooth convex solution of (\ref{EsLag}) yet.

But we can always arrange a quadratic function $Q=\frac{1}{2}K\left\vert
x\right\vert ^{2}+t$ touching the bounded continuous function $u$ from above
at an interior point $a$ in $B_{1}\left(  0\right)  .$ By the order
preservation of the rotation, $\bar{Q}=\frac{K-1}{2\left(  K+1\right)
}\left\vert x\right\vert ^{2}+t$ would touch $\bar{u}$ from above at the
corresponding interior point $\bar{a}$ in $\bar{\Omega}.$ It follows that
$D^{2}\bar{u}\left(  \bar{a}\right)  \leq\frac{K-1}{\left(  K+1\right)  }I<I.$
This contradiction shows that $D^{2}\bar{u}<I$ on $\bar{\Omega}.$

Step 3. We conclude by noting that now the original special Lagrangian graph
$(x,Du(x))$ is smooth and even analytic, and the effective Hessian bound in
Theorem \ref{theorem1} follows from Theorem 1.1 in \cite{CWY}. In fact, a
sharper bound follows from the a priori estimate Theorem 1.1 in \cite{WdY2},
since all proofs there go through when $\Theta\geq(n-2)\frac{\pi}{2}$ is
replaced by $\lambda_{i}\geq0$. An implicit Hessian bound would also follow
from a compactness argument, see e.g. \cite{L}.


\begin{thebibliography}{9999}                                                                                             %


\bibitem[BW]{BW}Brendle, Simon; Warren, Micah \emph{A boundary value problem
for minimal Lagrangian graphs.} J. Differential Geom. \textbf{84} (2010), no.
2, 267--287.

\bibitem[C]{C}Caffarelli, Luis A. \emph{Interior W}$^{2,p}$\emph{ estimates
for solutions of the Monge-Amp\`{e}re equation.} Ann. of Math. (2)
\textbf{131} (1990), no. 1, 135--150.

\bibitem[CC]{CC}Caffarelli, Luis A.; Cabr\'{e}, Xavier \emph{Fully Nonlinear
Elliptic Equations.} American Mathematical Society Colloquium Publications,
\textbf{43}. American Mathematical Society, Providence, RI, 1995.

\bibitem[CGM]{CGM}Caffarelli, Luis; Guan, Pengfei; Ma, Xi-Nan \emph{A constant
rank theorem for solutions of fully nonlinear elliptic equations.} Comm. Pure
Appl. Math. \textbf{60} (2007), no. 12, 1769--1791.

\bibitem[CCY]{CCY}Chau, Albert; Chen, Jingyi; Yuan, Yu \emph{Lagrangian mean
curvature flow for entire Lipschitz graphs II.} Math. Ann. \textbf{357}
(2013), no. 1, 165--183.

\bibitem[CW]{CW}Chen, Jingyi; Warren, Micah \emph{On the regularity of
Hamiltonian stationary Lagrangian submanifolds.} Adv. Math. \textbf{343}
(2019), 316--352.

\bibitem[CWY]{CWY}Chen, Jingyi; Warren, Micah; Yuan, Yu \emph{A priori
estimate for convex solutions to special Lagrangian equations and its
application.} Comm. Pure Appl. Math. \textbf{62}, no. 4 (2009), 583--595.

\bibitem[GT]{GT}Gilbarg, David; Trudinger, Neil S \emph{Elliptic Partial
Differential Equations of Second Order.} Reprint of the 1998 edition. Classics
in Mathematics. Springer-Verlag, Berlin, 2001.

\bibitem[HL]{HL}Harvey, Reese; Lawson, H. Blaine. Jr. \emph{Calibrated
geometry.} Acta Math. \textbf{148} (1982), 47--157.

\bibitem[H]{H}Heinz, Erhard \emph{On elliptic Monge-Amp\`{e}re equations and
Weyl's embedding problem.} J. Analyse Math. \textbf{7} (1959) 1--52.

\bibitem[Hi]{Hi}Hitchin, Nigel J. \emph{The moduli space of special Lagrangian
submanifolds.} Dedicated to Ennio De Giorgi. Ann. Scuola Norm. Sup. Pisa Cl.
Sci. (4) \textbf{25} (1997), no. 3-4, 503--515 (1998).

\bibitem[L]{L}Li, Caiyan \emph{A compactness approach to Hessian estimates for
special Lagrangian equations with supercritical phase.} Nonlin. Anal.
\textbf{187} (2019), 434-437.

\bibitem[Me]{Me}Mealy, Jack G. \emph{Calibrations on semi-Riemannian
manifolds.} Thesis (Ph.D.)--Rice University. 1989.

\bibitem[M]{M}Morrey, Charles B., Jr. \emph{On the analyticity of the
solutions of analytic non-linear elliptic systems of partial differential
equations. I. Analyticity in the interior.} Amer. J. Math. \textbf{80} (1958) 198--218.

\bibitem[NV]{NV}Nadirashvili, Nikolai; Vl\u{a}du\c{t}, Serge \emph{Singular
solution to special Lagrangian equations.} Ann. Inst. H. Poincar\'{e} Anal.
Non Lin\'{e}aire \textbf{27} (2010), no. 5, 1179--1188.

\bibitem[P]{P}Pogorelov, Aleksei Vasil'evich \emph{The Minkowski
multidimensional problem.} Translated from the Russian by Vladimir Oliker.
Introduction by Louis Nirenberg. Scripta Series in Mathematics. V. H. Winston
\& Sons, Washington, D.C.; Halsted Press [John Wiley \& Sons], New
York-Toronto-London, 1978.

\bibitem[R]{R}Rockafellar, R. Tyrrell \emph{Convex analysis.} Reprint of the
1970 original. Princeton Landmarks in Mathematics. Princeton Paperbacks.
Princeton University Press, Princeton, NJ, 1997.

\bibitem[U]{U}Urbas, John I. E. \emph{Regularity of generalized solutions of
Monge-Amp\`{e}re equations.} Math. Z. \textbf{197} (1988), no. 3, 365--393.

\bibitem[W]{W}Warren, Micah \emph{Calibrations associated to Monge-Amp\`{e}re
equations.} Trans. Amer. Math. Soc. \textbf{362} (2010), no. 8, 3947--3962.

\bibitem[WY1]{WY1}Warren, Micah; Yuan, Yu \emph{Hessian estimates for the
sigma-2 equation in dimension 3.} Comm. Pure Appl. Math. \textbf{62} (2009),
no. 3, 305--321.

\bibitem[WY2]{WY2}Warren, Micah; Yuan, Yu \emph{Explicit gradient estimates
for minimal Lagrangian surfaces of dimension two.} Math. Z. \textbf{262}
(2009), no. 4, 867--879.

\bibitem[WY3]{WY3}Warren, Micah; Yuan, Yu \emph{Hessian and gradient estimates
for three dimensional special Lagrangian equations with large phase.} Amer. J.
Math. \textbf{132} (2010), no. 3, 751--770.

\bibitem[WdY1]{WdY1}Wang, Dake; Yuan, Yu \emph{Singular solutions to special
Lagrangian equations with subcritical phases and minimal surface systems.}
Amer. J. Math. \textbf{135} (2013), no. 5, 1157--1177.

\bibitem[WdY2]{WdY2}Wang, Dake and Yuan, Yu \emph{Hessian estimates for
special Lagrangian equations with critical and supercritical phases in general
dimensions.} Amer. J. Math. \textbf{136} (2014), 481--499.

\bibitem[Y1]{Y1}Yuan, Yu \emph{A priori estimates for solutions of fully
nonlinear special Lagrangian equations.} Annales de l'Institut Henri
Poincar\'{e} Analyse Non Lin\'{e}aire \textbf{18}, no. 2 (2001), 261--270.

\bibitem[Y2]{Y2}Yuan, Yu \emph{A Bernstein problem for special Lagrangian
equations.} Invent. Math. \textbf{150} (2002), 117--125.

\bibitem[Y3]{Y3}Yuan, Yu \emph{Global solutions to special Lagrangian
equations.} Proc. Amer. Math. Soc. \textbf{134} (2006), no. 5, 1355--1358.
\end{thebibliography}
\end{document}